\documentclass[]{article}
\usepackage{graphicx}
\usepackage{amsmath} 
\usepackage{amssymb}
\usepackage{amsthm}
\usepackage{breqn}
\usepackage{indentfirst}
\usepackage{parskip}
\usepackage{wasysym}
\usepackage{bbm}
\usepackage[round]{natbib}
\usepackage{float}
\usepackage{mathtools}
\usepackage{authblk}

\theoremstyle{definition}
\newtheorem{Definition}{Definition}[section]
\newtheorem{Theorem}{Theorem}[section]
\newtheorem{Proposition}{Proposition}[section]
\newtheorem{Corollary}{Corollary}[section]
\newtheorem{Lemma}{Lemma}[section]
\newtheorem{Remark}{Remark}[section]

\author[*]{Yufan Li}
\author[**]{Jeffery Rosenthal}
\affil[*]{John A. Paulson School Of Engineering And Applied Sciences, Harvard University, USA. Electronic address: \texttt{yufan\_li@g.harvard.edu}; Corresponding author}
\affil[**]{Department of Statistics, University of Toronto, Canada. Electronic address: \texttt{jeff@math.toronto.edu}}

\begin{document}

\title{A Divergent Random Walk on Stairs}

\maketitle

\begin{abstract}
	We consider a state-dependent, time-dependent, discrete random walks $X_t^{\{a_n\}}$ defined on natural numbers $\mathbb{N}$ (bent to a ``stair'' in $\mathbb{N}^2$) where the random walk depends on input of a positive deterministic sequence $\{a_n\}$. This walk has the peculiar property that if we set $a_n$ to be $+\infty$ for all $n$, it converges to a stationary distribution $\pi(\cdot)$;  but if $a_n$ is uniformly bounded (over all $n$) by any upper bound $a \in (0,\infty)$, this walk diverges to infinity with probability 1. It is thus interesting to consider the intermediate case where $a_n<\infty$for all $n$ but eventually tends to $+\infty$. 
	
	\citep{latuszynski2013adaptive} first defined this walk and conjectured that a particular choice of sequence $\{a_n\}$ exists such that (i) $a_n \to \infty$ and, (ii) $P(X_t^{\{a_n\}} \to \infty )=1$. They managed to construct a sequence $\{a_n\}$ that satisfies (i) and $P(X_t^{\{a_n\}}\to \infty)>0$, which is weaker than (ii).  
	
	In this paper, we obtain a stronger result: for any $\sigma<1$, there exists a choice of $\{a_n\}$ so that $P(X_t\to \infty)\ge \sigma$. Our result does not apply when $\sigma=1$, the original conjecture remains open. We record our method here for technical interests. 
\end{abstract}

\tableofcontents

\section{Random Walk on Stairs and Conjecture}
This problem was first proposed in \cite{latuszynski2013adaptive} where one may find the original motivation (related to adaptive Gibbs sampler in Monte Carlo literature), precise statement of the unproven conjecture and their attempt. We will re-state the problem here. 

Let $\mathcal{N}=\{1,2....\}$ and let the state space $\mathcal{X}=\{(i,j) \in \mathcal{N} \times \mathcal{N}, i=j \text{ or } i=j+1\}$. This describes the ``stair'' in $\mathbb{R}^2$ as shown in Figure \ref{fig:screenshot001}. Define a probability distribution $\pi(\cdot)$ on $\mathcal{X}$ by $\pi(i,j) \propto j^{-2}, \forall (i,j) \in \mathcal{X}$. Now we are ready to specify random walk on stairs mentioned in abstract. 

\begin{figure}[H]
	\centering
	\includegraphics[width=0.7\linewidth]{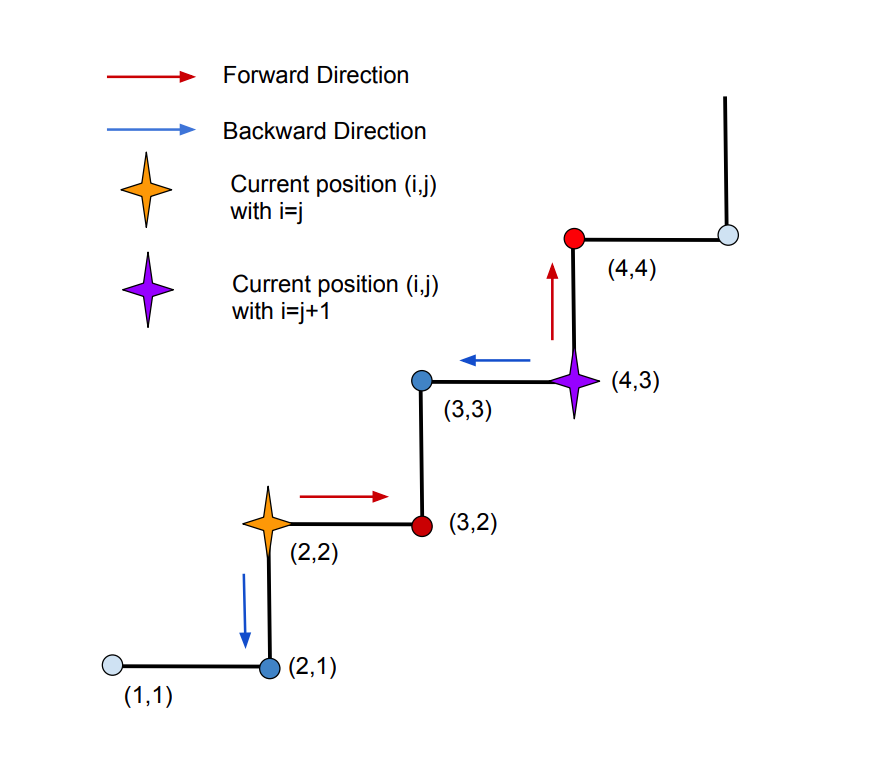}
	\caption{Random Walk on Stairs}
	\label{fig:screenshot001}
\end{figure}

\begin{Definition}[Random Walk on Stairs] \label{RWoS}
	We illustrate the definition in Figure \ref{fig:screenshot001}. With respect to a fixed deterministic, positive sequence $\{a_n\}$ satisfying $8\le a_n \nearrow \infty,$ we define a random walk $X_t, t=0,1,2,...$ (denoted as $X_t^{\{a_n\}}$ in abstract) on the stair $\mathcal{X}$ by first defining $X_0=(1,1)$ and then by the induction as follows: 
	
	\begin{itemize}
		\item If $X_{t-1}=(i,j)$ such that $i=j$ (yellow star),
		\begin{itemize}
			\item with probability $\frac{1}{2}-\frac{4}{a_n}$, the walk choose to move along the backward direction (blue arrow, parallel to $y-$axis) where it further choose to move backwards one step to the new state $X^\star$ (colored blue) with probability $p^\star=\frac{\pi(X^\star)}{\pi(X^\star)+\pi(X_{n-1})}$ or stay put at current location with probability $1-p^\star$;
			\item with probability $\frac{1}{2}+\frac{4}{a_n}$, the walk choose to move along the forward direction (red arrow, parallel to $x-$axis) where it further choose to move forwards one step to the new state $X^\star$ (colored red) with probability $p^\star=\frac{\pi(X^\star)}{\pi(X^\star)+\pi(X_{n-1})}$ or stay put at current location with probability $1-p^\star$;
		\end{itemize}
		
		\item If $X_{t-1}=(i,j)$ such that $i=j+1$ (purple star),
		\begin{itemize}
			\item with probability $\frac{1}{2}+\frac{4}{a_n}$, the walk choose to move along the backward direction (blue arrow, parallel to $x$-axis) where it further choose to move backwards one step to the new state $X^\star$ (colored blue) with probability $p^\star=\frac{\pi(X^\star)}{\pi(X^\star)+\pi(X_{n-1})}$ or stay put at current location with probability $1-p^\star$;
			\item with probability $\frac{1}{2}-\frac{4}{a_n}$, the walk choose to move along the forward direction (red arrow, parallel to $y$-axis) where it further choose to move forwards one step to the new state $X^\star$ (colored red) with probability $p^\star=\frac{\pi(X^\star)}{\pi(X^\star)+\pi(X_{n-1})}$ or stay put at current location with probability $1-p^\star$;
		\end{itemize}

	\end{itemize}
	
\end{Definition}
\bigskip
\begin{Remark}
	The assertion made in the abstract can be easily verified. As $a_n=\infty$ for all $n$, the random walk on stair degenerates to a classic random-scan Gibbs sampler that converges to $\pi$. On the other hand, if $a_n=a<\infty$ (uniformly bounded), the random walk may be stochastically lower bounded by a simple random walk on $\mathbb{N}$ biased towards infinity. 
\end{Remark}

Now, we are ready to formally state the conjecture: \\

\begin{Definition}[The Open Conjecture]
	Does there exist a choice of $a_n \to \infty, a_n \ge 8$ so that the random walk defined in Definition \ref{RWoS} converges to $\infty$ with probability 1?
\end{Definition}
\bigskip 
\begin{Remark}
	As stated in the abstract, \cite{latuszynski2013adaptive} found a sequence $\{a_n\}$ such that $P(X_t \to \infty)>0$. 
\end{Remark}

\section{Our Result}

To the problem posed in last section, we are able to state the following:\\

\begin{Theorem}[Main Result]\label{mainres}
	For any fixed $\sigma \in [0,1)$, there exists a choice of ${a}_n$ such that $P({X}_n \rightarrow \infty)>\sigma$.
\end{Theorem}

That is, for any fixed $\sigma<1$, we are able to construct a sequence $a_n$ so that the random walk (with parameter of $a_n$) diverges to infinity with probability larger than $\sigma$.

\section{Proof Strategy and Intuition}
An important observation is that when $a_n=\infty$, the random walk defined in last section would converge to $\pi$ and will never diverge to infinity. On the contrary, if $a_n=a<\infty$, the process will diverge to infinity with probability 1. Therefore, the intuition is that we should ``slow down'' growth rate of  $a_n$ sufficiently to obtain the desired divergent result. The overall strategy of our proof is to divide steps of the process into different ``phases'' where $a_n$ stays constant in each phase. For example, we would categorize first some $N_1$ step as phase $1$, and some $N_2$ steps as phase $2$, and so forth. The hope is that we would be able to find appropriate duration of each phase, i.e. some choice of $\{N_i\}$, so that the probability of the process residing anywhere below the ``height'' $y=i$ is small after phase $i$.

\section{Proof of the Main Theorem}
In this section, we give a proof of the main theorem. 

\subsection{``Flatten the Stair''}
\bigskip
\begin{Definition}[``Flatten the Stair'']
	Let $S_n$ denote the ``distance'' from $X_n$ to starting position $(1,1)$. If $X_{n}=(x_{n},x_{n})$, $S_n=2(x_n-1)$; if $X_{n}=(x_{n},x_{n}-1)$, $S_n=2(x_n-1)-1=2x_n-3$.
\end{Definition}
\bigskip
\begin{Remark}
	Note that that the random walk on stair is just a usual random walk $S_n$ on natural number with each step following distribution specified in \ref{Eachstp}. We just need to show that $S_n$ goes to infinity. 
\end{Remark}

First we first note the following Lemma:
\bigskip
\begin{Lemma}[Distribution of $S_n-S_{n-1}$]\label{Eachstp}
	If $X_{n}=(x_{n},x_{n})$, the distribution of $S_{n+1}-S_n$, taking value $\{-1,0,1\}$, is
	\begin{equation}\label{Sup}
	\bigg((\frac{1}{2}-\frac{4}{a_n})\frac{x_{n}^2}{x_{n}^2+(x_{n}-1)^2},1-(\frac{1}{2}-\frac{4}{a_n})\frac{x_{n}^2}{x_{n}^2+(x_{n}-1)^2}-(\frac{1}{4}+\frac{2}{a_n}),\frac{1}{4}+\frac{2}{a_n}\bigg)
	\end{equation}
	
	If $X_{n}=(x_{n},x_{n}-1)$, the distribution of $S_{n+1}-S_n$, taking value $\{-1,0,1\}$, is
	\begin{equation}\label{Sdown}
	\bigg(\frac{1}{4}-\frac{2}{a_n},1-(\frac{1}{4}-\frac{2}{a_n})-(\frac{1}{2}+\frac{4}{a_n})\frac{(x_{n}-1)^2}{x_{n}^2+(x_{n}-1)^2},(\frac{1}{2}+\frac{4}{a_n})\frac{(x_{n}-1)^2}{x_{n}^2+(x_{n}-1)^2}\bigg)
	\end{equation}
\end{Lemma}
\begin{proof}
	This follows directly from definition of the random walk in Definition \ref{RWoS} and that $S_n$ is just a ``flattened version'' of the walk.
\end{proof}

\subsection{Define Phases and $a_n$}
Fix any $\sigma \in [0,1)$. The main theorem of this paper is that if one defines $a_n$ as the following ($a_n$ depends on $\sigma$), then the walk diverges to infinity with probability larger than $1-\sigma$. Our construction partition entire time line (i.e. $\mathbb{N}$) into countably many phases with phase $i$ of length $N_i-N_{i-1}$, and that $a_n$ is defined to be the same number on each phase. 

At this point, some of definitions below may not be well-motivated, i.e. it is not obvious why we choose to define them in this particular fashion. But the point is that $a_n$ and thus the random walk on stairs for each $\sigma$ is well-defined and satisfies requirements posed by the conjecture. The motivation will become clear as the proof moves forward in the following sections.
\\

\begin{Definition}[Define Parameters Required to Define $N_i$]
	Fix any $M\in \mathbb{N}$ such that for all $m \ge M-2$
	\[0.01\cdot(m)-2 \sqrt{K\cdot (m)\ln(m)}\ge 4\]
	Note that $M$ exists for any fixed natural number $K$ because the first term is of order $m=\sqrt{m\cdot m}$ and the second term is of smaller order $\sqrt{m \ln(m)}$. 
\end{Definition}

\bigskip
\begin{Definition}[Define Duration of Each Phases $i$]\label{DefNi}
	Define sequence $N_i$ inductively such that $N_0=0$, $N_1=M_0(\sigma,M)$ and $N_{i}-N_{i-1}=M-2+2(i-2), i=2,3,...$; $M_0(\sigma,M)$ could be any natural number satisfying that for all $m \ge M_0(\sigma,M)$, $$P(\sum_{j=1}^{m} B_j>M)>1-\sigma$$ where $B_j$ are i.i.d Bernoulli variables with $P(B_j=1)=1/5$ and $P(B_j=0)=4/5$. Such $M_0(\sigma,M)$ exists by Central Limit Theorem. 
\end{Definition}

\bigskip
\begin{Definition}[Define choice of $a_n$ for fixed $\sigma$]\label{defan}
	Define $a_n$ as follows:
	\begin{equation}
	\begin{cases}
	a_n=8, \; \textrm{if}\; 0 \le n < N_1 \\
	a_n=8\frac{2i^2+1-2i}{2i-1+0.1}-0.001, \; \textrm{if}\; N_{i-1} \le n < N_{i}, \forall i\ge 2
	\end{cases}
	\end{equation}
\end{Definition}

\bigskip
\begin{Lemma}[Choice $a_n$ Satisfies Requirement]
	(i) $a_n \to \infty$;
	(ii) $a_n \ge 8$.
\end{Lemma}
\begin{proof}
	(i) is true since $N_i$ is finite for each $i$; (ii) is a just an algebra exercise: it is easy to show that $a_n$ increases for $n \ge 2$ and $a_2>8$.
\end{proof}

\subsection{Define Auxiliary Variables}
The idea is that, for $i$-th phase, we construct a auxiliary variable $Z_i$ upon some condition so that $\mathbb{E}(Z_i)\ge 0.01$ (that is, $Z_i$ is ``increasing on average'') and it is stochastically smaller for each step $S_n-S_{n-1}$ within $i-$th phase. Note that $x_n \ge i$ condition is crucial--that is, we must ensure that at phase $i$, distance between $X_n$ and $y-$axis is at least $n$ for an auxiliary $Z_i$ to exist. \\

\begin{Proposition}[Define $Z_i$ for Phase $i$]
	For each $i \ge 2$, if $x_n \ge i$ and $N_{i-1}\le n <N_{i}$, there exists a sequence of i.d.d random variable $\{Z_i\}$ such that $Z_i$ is stochastically smaller than $S_{n+1}-S_{n}$ for each $n \in [N_{i-1}, N_{i})$ and $Z_{i}$ take value $\{-1,0,1\}$ and $\mathbb{E}(Z_i)\ge 0.01$
\end{Proposition}
\begin{proof}
	Since $x_n\ge i$, and $a_n$ is constant for $N_{i-1}\le n <N_{i}$
	
	\begin{equation}\label{InSup}
	\frac{1}{4}+\frac{2}{a_n}>(\frac{1}{2}+\frac{4}{a_n})\frac{(x_n-1)^2}{{x_n}^2+(x_n-1)^2}\geq(\frac{1}{2}+\frac{4}{a_n})\frac{(i-1)^2}{{i}^2+(i-1)^2}
	\end{equation}
	\begin{equation}\label{InSdown}
	\frac{1}{4}-\frac{2}{a_n}<(\frac{1}{2}-\frac{4}{a_n})\frac{(x_n)^2}{(x_n)^2+(x_n-1)^2} \leq (\frac{1}{2}-\frac{4}{a_n})\frac{i^2}{i^2+(i-1)^2}
	\end{equation}
	where we have used the facts that $f(x)=\frac{(x-1)^2}{x^2+(x-1)^2}$ is monotone increasing on $x\ge1$ and that $f(x)=\frac{(x)^2}{x^2+(x-1)^2}$ is monotone decreasing on $x\ge1$. 
	
	Solve the following inequality: 
	\begin{equation}\label{upeq}
	(\frac{1}{2}+\frac{4}{a_n})\frac{(i-1)^2}{{i}^2+(i-1)^2}-(\frac{1}{2}-\frac{4}{a_n})\frac{i^2}{i^2+(i-1)^2}>0.1
	\end{equation}
	We obtain:
	\[a_n<8\frac{2i^2+1-2i}{2i-1+0.1}\]
	This means that with our choice of $a_n$, for all $n \in [N_{i-1}, N_{i})$
	\[(\frac{1}{2}+\frac{4}{a_n})\frac{(i-1)^2}{{i}^2+(i-1)^2}-(\frac{1}{2}-\frac{4}{a_n})\frac{i^2}{i^2+(i-1)^2}>0.1\]

	We want to define distribution of $Z_i$ as $(c_i,1-c_i-b_i,b_i)$ where for $n \in [N_{i-1}, N_{i})$
	
	\begin{equation}
	c_i=(\frac{1}{2}-\frac{4}{a_n})\frac{i^2}{i^2+(i-1)^2}+0.0001
	\end{equation}
	\begin{equation}
	b_i=(\frac{1}{2}+\frac{4}{a_n})\frac{(i-1)^2}{{i}^2+(i-1)^2}-0.0001
	\end{equation}
	Note that here $a_n$ depends only on $i$ (check Definition \ref{defan}) and that $(c_i,1-c_i-b_i,b_i)$ is a valid probability distribution since for all $i \ge 2$, 
	\[(\frac{1}{2}-\frac{4}{a_n})\frac{i^2}{i^2+(i-1)^2}\le 4/5 \cdot 1/2=2/5\]
	and since $a_n>10$ for $i\ge 2$
	\[(\frac{1}{2}+\frac{4}{a_n})\frac{(i-1)^2}{{i}^2+(i-1)^2}<1/2 \cdot(1/2+4/10)=9/20\]
	
	Now, compare Equation (\ref{Sdown}), (\ref{Sup}) with Inequality (\ref{InSdown}), (\ref{InSup}) to see that $Z_i$ indeed is stochastic smaller than $S_n-S_{n-1}$ for each phase $i$. 
 	And $\mathbb{E}(Z_i)>0.1-0.0002=0.0998>0.01$ by Inequality (\ref{upeq}).
\end{proof}

\subsection{``Divergent Speed'' of Auxiliary Random Walk}
Define $I_{i,m}:=\sum_{j=1}^{m}Z_{i}$. Since $Z_i$ is strictly bounded by $[-1,1]$, by Hoeffding's inequality, for any $t>0$,
\begin{equation}
P(|I_{i,m}-E(I_{i,m})|\geq t)\leq 2\exp(-\frac{t^2}{2m})
\end{equation}

For any $n>1$, let the "bound" $t$ be $t_m=2\sqrt{Kn\ln(n)}$ for each $n$. Then the probability of "exceeding the bound" for each $m$ is

\begin{equation}
P\bigg(|I_{i,m}-0.01m|\geq 2\sqrt{Km\ln(m)}\bigg)\leq \frac{2}{m^{2K}} 
\end{equation}

Therefore,
\begin{equation}
P\bigg(I_{i,m}>0.01m- 2\sqrt{Km\ln(m)}\bigg)>1- \frac{2}{m^{2K}} 
\end{equation}

\subsection{Inductive Events for Condition for Constructing $Z_i$}
Since eventually we seek to lower bound $S_n-S_{n-1}$ during phase $i$ with the i.i.d auxiliary variables $Z_i$, we must ensure that $x_n\ge i$ during phase $i$ even under the worst case scenario where $S_n-S_{n-1}=-1$ for the entire duration of phase $i$. This requires that we set up a stronger inductive hypothesis in the sense that there is some ``overshoot''. 

First, we will consider phase 1. During phase 1, that is $n<N_1$, $a_n=8$. Check Equation (\ref{Sup}, (\ref{Sdown}), $S_{n+1}-S_{n}\ge 0$ with probability 1 during phase 1, and that $P(S_n-S_{n-1}=1)\ge 1/5$. For a particular choice of $M, \sigma$, the definition of $N_1=M_0(\sigma, M)$ in Definition \ref{DefNi} ensures that $P(S_{N_1}>M)>(1-\sigma/2)$. 

Now we will proceed to set up the induction. Consider the following events:

	\[\Omega_1=\{S_{M_0}>M\}\]
	\[\Omega_i=\{S_{N_i} \ge M+4(i-1)\}, \forall i \ge 2\]
	We need to first prove some lemmas.
	\medskip
	\begin{Lemma} [Satisfy Condition for Constructing $Z_i$]
		For each $i \ge 2$, under event $\cap_{j=1}^{i-1}\Omega_{i}$, we have that $x_n \ge i$ for all $n \in [N_{i-1}, N_i)$.
	\end{Lemma}
	\begin{proof}
		We know that under event $\cap_{j=1}^{i-1}\Omega_{i}$, $S_{N_{i-1}} \ge M+4(i-2)$. We now claim that for $N_{i-1}\le n< N_i$, $x_n \ge i$: assume the worst case where $S_n-S_{n-1}=-1$ always for $N_{i-1}\le n< N_i$; since $N_i-N_{i-1}=M-2+2(i-2)$, we know that the smallest $S_n$ is achieved at $S_{N_i}$: 
		$$S_{N_{i}}\ge S_{N_{i-1}} - (M-2+2(i-2)) \ge M+4(i-2)-M+2-2i+4=2i-2.$$
		This implies that the smallest value possible for $x_n, n \in [N_{i-1}, N_i)$ is $x_{N{i}}\ge i$. So $x_n \ge i, n \in [N_{i-1}, N_i)$. 
	\end{proof}

\subsection{Stochastic Domination over $Z_i$}
	Before we prove the next lemma, we restate the following theorem concerning monotone coupling and stochastic domination:
	\bigskip
	\begin{Theorem}
		The real random variable $W$ is stochastically larger than $V$ if and only if there is a coupling between $W,V$ such that
		\[P(W\ge V)=1\]
	\end{Theorem}
	\textbf{Remark. } This means that if $W_n$ is stochastically larger than $V_n$ respectively for all $n>1$, we may find a coupling for each $n$ between $W_n$ and $V_n$ such that $P(\sum{W_n}\ge \sum{V_n})=1$. For us, we intend to, for phase $i$, choose $W_n$ as $S_n-S_{n-1}$ and $V_n$ as i.i.d $Z_{n,i}\sim Z_i$. 
	\bigskip
	\begin{Lemma}[Bound $S_n$ using Stochastic Domination]\label{leylemma}
		For each $i \ge 2$, for all $n \in [N_{i-1}, N_i)$,
		\[P\{S_{N_i}\ge M+4(i-1)|\cap_{j=1}^{i-1}\Omega_{i}\}\ge 1-\frac{2}{(M-2+2(i-2))^{2K}}\]
	\end{Lemma}
	\begin{proof}
		Note that $N_i-N_{i-1}=M-2+2(i-2)\ge M-2,\forall i\ge2$. And from definition of $M$, we know that for all $m \ge M-2$
		\begin{equation}\label{inst1}
			0.01m- 2\sqrt{Km\ln(m)}\ge 4.
		\end{equation}

		We have also shown that 
		\begin{equation}\label{inst2}
		P\bigg(I_{i,m}>0.01m- 2\sqrt{Km\ln(m)}\bigg)>1- \frac{2}{m^{2K}} 
		\end{equation}
		From the previous lemma, we know that for $n \in [N_{i-1}, N_i)$ we may couple $S_{n+1}-S_{n}$ with $Z_i$ since $x_n\ge i$. This gives us that
		\[P\{S_{N_i} \ge M+4(i-1)|\cap_{j=1}^{i-1}\Omega_{i}\} \ge P\{S_{N_{i-1}}+I_{i,N_i-N_{i-1}} \ge M+4(i-1)|\cap_{j=1}^{i-1}\Omega_{i}\}\]
		\[\ge P\bigg\{I_{i,M-2+2(i-2)} \ge M+4(i-2)-M-4(i-2)=4\bigg\}\]
		\[\ge P\bigg\{I_{i,M-2+2(i-2)} \ge 
		\]\[0.01(M-2+2(i-2))-2 \sqrt{K(M-2+2(i-2))\ln{(M-2+2(i-2))}})\bigg\}\]
		\[> 1- \frac{2}{(M-2+2(i-2))^{2K}}\]
		where the stochastic domination is used at the first inequality, the second inequality holds because we condition on event $\cap_{j=1}^{i-1}\Omega_{i}$ and we dropped conditioning afterwards since $I_{i,M-2+2(i-2)}$ are just sum of i.i.d. $Z_i$'s, the third inequality is by using Inequality (\ref{inst1}) above, and the forth inequality is by using Inequality (\ref{inst2}) above. 
	\end{proof}

	\begin{Corollary}\label{cor}
		From Lemma \ref{leylemma}, we have
		\[P(S_n \to \infty)\ge (1-\sigma/2)\cdot \Pi_{j=0}^\infty(1-\frac{2}{(M-2+2j)^{2K}})\]
	\end{Corollary}
	\begin{proof}
		Lemma \ref{leylemma} essentially provides us with the following:
		\[P\{\Omega_i|\cap_{j=1}^{i-1}\Omega_j\} \ge 1-\frac{2}{(M-2+2j)^{2K}}\]
		Since $\cap_{j=1}^{\infty}\Omega_j$ implies that $\{S_n \to \infty, n \to \infty\}$, the claim follows by induction. 
	\end{proof}

\subsection{Main Theorem}
Let $K=1$. Fix $\sigma>0$. Note that we may choose $M$ to be arbitrary larger and Corollary \ref{cor} will hold for all the random walks defined with these $M$. Note in addition the fact that $\lim_{M\to \infty}\Pi_{j=0}^\infty(1-\frac{2}{(M-2+2j)^{2}})=1.$ The main theorem is thus proved.

\bibliography{learningnotes}	
\end{document}